\documentclass[graybox,footinfo]{svmult}

\smartqed
\usepackage{mathptmx}       
\usepackage{helvet}         
\usepackage{courier}        
\usepackage{type1cm}        
\usepackage{graphicx}       

\usepackage{array,colortbl}
\usepackage{amsmath,amsfonts,amssymb,bm} 

\newcommand{\R}{\mathbb{R}} 
\newcommand{\N}{\mathbb{N}} 
\newcommand{\setu}{\mathfrak{u}}
\newcommand{\setv}{\mathfrak{v}}

\usepackage{xcolor}
\xdefinecolor{dblue}{rgb}{0,0,0.7}
\xdefinecolor{dgreen}{rgb}{0,0.6,0} 
\xdefinecolor{dred}{rgb}{0.7,0,0}

\newcommand{\EE}{\mathbb{E}}
\newcommand{\PP}{\mathbb{P}}
\DeclareMathOperator{\cov}{Cov}
\DeclareMathOperator{\var}{Var}
\newcommand{\bsu}{\setu}    
\newcommand{\bsv}{\setv}    

\usepackage{microtype} 

\usepackage[colorlinks=true,linkcolor=black,citecolor=black,urlcolor=black]{hyperref}
\usepackage{bookmark}
\makeatletter 
  \providecommand*{\toclevel@author}{999}
  \providecommand*{\toclevel@title}{0}
\makeatother

\begin{document}

\title*{On ANOVA decompositions of kernels and Gaussian random field paths}
\titlerunning{On ANOVA decompositions of kernels and Gaussian random field paths} 
\author{D. Ginsbourger \and O. Roustant \and D. Schuhmacher \and N. Durrande \and N. Lenz}
\authorrunning{D. Ginsbourger \and O. Roustant \and D. Schuhmacher \and N. Durrande \and N. Lenz} 
\institute{
David Ginsbourger \and Nicolas Lenz 
\at Department of Mathematics and Statistics, University of Bern, Alpeneggstrasse 22, CH-3012 Bern, Switzerland. 
\email{ginsbourger@stat.unibe.ch}
\email{nicolas.lenz@gmail.com}
\and 
Olivier Roustant \and Nicolas Durrande 
\at \'Ecole Nationale Sup\'erieure des Mines, FAYOL-EMSE, LSTI, F-42023 Saint-Etienne, France. 
\email{roustant@emse.fr}
\email{durrande@emse.fr}
\and 
Dominic Schuhmacher 
\at Institut f\"ur Mathematische Stochastik,
Georg-August-Universit\"at G\"ottingen,
Goldschmidtstra{\ss}e 7, D-37077 G\"ottingen, Germany. 
\email{dominic.schuhmacher@mathematik.uni-goettingen.de}
}
\maketitle

\abstract{
The FANOVA (or ``Sobol'-Hoeffding'') decomposition of multivariate functions has been used for high-dimensional model representation and global sensitivity analysis. 
When the objective function $f$ has no simple analytic form and is costly to evaluate, a practical limitation is that computing FANOVA terms may be unaffordable due to numerical integration costs. 
Several approximate approaches relying on random field models have been proposed to alleviate these costs, where $f$ is substituted by a (kriging) predictor or by conditional simulations. 
In the present work, we focus on FANOVA decompositions of Gaussian random field sample paths, and we notably introduce an associated kernel decomposition (into $2^{2d}$ terms) called KANOVA. 
An interpretation in terms of tensor product projections is obtained, and it is shown that projected kernels control both the sparsity of Gaussian random field sample paths and the dependence structure between FANOVA effects. Applications on simulated data show the relevance of the approach for designing new classes of covariance kernels dedicated to high-dimensional kriging.
}

\section{Introduction: Metamodel-based Global Sensitivity Analysis}\label{sec:1}

Global Sensitivity Analysis (GSA) is a topic of importance for the study of complex systems as it aims at uncovering among many candidates which variables and interactions thereof 
are influential with respect to some response of interest. 
FANOVA (Functional ANalysis Of VAriance) \cite{Hoeffding1948, 
Sobol'1969, Efron.Stein1981, ant84}
has become commonplace for decomposing a real-valued function $f$ of $d$-variables into a sum of $2^{d}$ functions (a.k.a. \textit{effects}) of increasing dimensionality, and quantifying the influence of each variable or group of variables through non-negative indices summing up to one, the celebrated \textit{Sobol'  indices} \cite{sobol2001global, saltelli2008global}.
In practice $f$ is rarely known analytically and a number of statistical procedures have been 
proposed for estimating Sobol' indices based on a finite sample of evaluations of $f$; see \cite{Janon.etal2013} and the references therein. 
Alternatively, a pragmatic approach to GSA, when the evaluation budget is drastically limited by computational cost or time, is to first approximate $f$ by employing some class of surrogate models (e.g., regression, neural nets, splines, wavelets, kriging; see \cite{Touzani2011} for an overview) and then to perform the analysis on the obtained cheap-to-evaluate surrogate model. 
Here we focus essentially on kriging and Gaussian random field (GRF) models, with an emphasis on the interplay between covariance kernels and FANOVA decompositions of corresponding centred GRF sample paths. 

While screening and GSA relying on kriging have been used for at least two decades \cite{wel:buc:sac:wyn:mit:mor92}, probabilistic GSA in the Bayesian set-up seems to originate in \cite{Oakley2004}, where posterior effects and related quantities were derived under a GRF prior. 
Later on, posterior distributions of Sobol' indices were investigated in \cite{Marrel2009} relying on conditional simulations, an approach revisited and extended to multi-fidelity computer codes in \cite{LeGratiet.etal2014}. 
From a different perspective, FANOVA-graphs were used in \cite{Muehlenstaedt.etal2012} to incorporate GSA information into a kriging model, and a special class of kernels was introduced in \cite{jmva_nico} for which Sobol' indices of the kriging predictor are analytically tractable. 
Moreover, a class of kernels leading to centred GRFs with additive paths has been discussed in \cite{afst_nico}, and FANOVA decompositions of GRFs and their covariance were touched upon in \cite{Lenz2013} where GRFs with ortho-additive paths were introduced. 
More recently, a variant of the kernel investigated in \cite{jmva_nico} was revisited in \cite{Chastaing.LeGratiet} with a focus on GSA with dependent inputs, and a class of kernels related to ANOVA decompositions was studied in \cite{Duvenaud2011additif, Duvenaud2014}. 
In a different setting, GRF priors have been used for Bayesian FANOVA when the responses are curves \cite{Kaufman.Sain2010}.  

In the present paper we investigate ANOVA decompositions both for
(symmetric positive definite) kernels and for associated centred GRFs.
We show that under standard integrability conditions, s.p.d. kernels can be decomposed into $4^d$ terms that govern the joint distribution of the $2^d$ terms of the associated GRF FANOVA decomposition. 
This has some serious consequences in kriging-based GSA, as for instance the choice of a sparse kernel induces almost sure sparsity of the associated GRF paths, and that such phenomenon cannot be compensated by data acquisition.  

\section{Preliminaries and notation}\label{sec:2}

\textbf{FANOVA}. We focus on measurable $f: D \subseteq \R^{d} \longrightarrow \R$ ($d\in \N\backslash \{0\}$). In FANOVA with independent inputs, $D$ is typically assumed to be of the form $D=\prod_{i=1}^{d} D_{i}$ for some measurable subsets $D_{i}\in \mathcal{B}(\R)$, where each $D_{i}$ is endowed with a probability measure $\nu_{i}$ and $D$ is equipped with the product measure $\nu=\bigotimes_{i=1}^{d}\nu_{i}$.  
Assuming further that $f$ is square-integrable w.r.t. $\nu$, 
$f$ can be expanded into as sum of $2^d$ terms indexed by the subsets $\bsu \subseteq I=\{1,\dots, d\}$ of the $d$ variables,
\begin{equation}
\label{eq1}
f=\sum_{\bsu \subseteq I}f_{\bsu},
\end{equation}
where $f_{\bsu} \in \mathcal{F}=\mathrm{L}^2(\nu)$ 
depend only on the variables $x_j$ with $j \in \bsu$ (up to an a.e. equality, as all statements involving $\mathrm{L}^2$ from Equation~\eqref{eq1} on). 
Uniqueness of this decomposition is classically guaranteed by imposing that $\int f_{\bsu} \, \nu_{j}(\mathrm{d}x_{j}) = \mathbf{0}$ 
for every $j \in \bsu$,
in which case the FANOVA \textit{effects} $f_{\bsu}$ can be expressed in closed form as 
\begin{equation}
\label{eq2}
f_{\bsu}: \vec{x} \in D \longrightarrow 
f_{\bsu}(x_{1}, \dots, x_{d})=\sum_{\bsu' \subseteq \bsu}(-1)^{\vert \bsu\vert-\vert \bsu'\vert}\int 
f(x_{1}, \dots, x_{d}) \; \nu_{-\bsu'}(\mathrm{d} \vec{x}_{-\bsu'}),
\end{equation}
\noindent
where $\nu_{-\bsu'}=\bigotimes_{j \in I \setminus \bsu'} \nu_{j}$ and $\vec{x}_{-\bsu'}=(x_{i})_{i \in I \setminus \bsu'}$.
As developed in \cite{kuo:slo:was:woz10}, Equation~\eqref{eq2} is a special case of a decomposition relying on commuting projections.
Denoting by $P_{j}: f \in \mathcal{F} \longrightarrow \int f \mathrm{d} \nu_{j}$ the orthogonal projector onto the subspace $\mathcal{F}_{j}$ of $f\in \mathcal{F}$ not depending on $x_{j}$,  
the identity on $\mathcal{F}$ can be expanded as 
\begin{equation}
I_{\mathcal{F}} = \prod_{j=1}^d \big[ (I_{\mathcal{F}} -P_j) + P_j \big] = 
\sum_{\bsu \subseteq I } 
\biggl( \prod_{j \in \bsu} (I_{\mathcal{F}} -P_{j}) \biggr)
\biggl( \prod_{j \in I \setminus \bsu} P_{j} \biggr).
\end{equation}
FANOVA effects appear then as images of $f$ under the orthogonal projection operators onto the associated subspaces $\mathcal{F}_{\bsu}=\left(\bigcap_{j \notin \bsu} \mathcal{F}_{j}\right) \cap \bigl(\bigcap_{j \in \bsu} \mathcal{F}_{j}^{\perp}\bigr)$, i.e.\ we have that $f_{\bsu} = T_{\bsu}(f)$, where $T_{\bsu} = \left( \prod_{j \in \bsu} (I_{\mathcal{F}}-P_{j}) \right) \, \left( \prod_{j \notin \bsu} P_{j} \right)$. 
Finally, the squared 
norm of $f$ decomposes by orthogonality as 
$\|f\|^2 = \sum_{\bsu \subseteq I} \|T_{\bsu}(f)\|^2$ and the influence of each (group of) variable(s) on $f$ can be quantified via the Sobol' indices
\begin{equation} \label{eq:gensobol}
  S_{\bsu}(f) = \frac{\|T_{\bsu}(f-T_{\emptyset}(f))\|^2}{\|f-T_{\emptyset}(f)\|^2} = \frac{\|T_{\bsu}(f)\|^2}{\|f-T_{\emptyset}(f)\|^2}, \quad \bsu \neq \emptyset.
\end{equation}

\noindent
\textbf{Gaussian random fields (GRFs)}. A random field indexed by $D$ is a collection of random variables $Z=(Z_{\vec{x}})_{\vec{x}\in D}$ defined on a common probability space $(\Omega, \mathcal{A}, \mathbb{P})$. 
The random field is called a Gaussian random field (GRF) if
$(Z_{\vec{x}^{(1)}}, \dots, Z_{\vec{x}^{(n)}})$ is $n$-variate
normally distributed for any $\vec{x}^{(1)}, \dots, \vec{x}^{(n)} \in
D$ $(n\geq 1)$. The distribution of $Z$ is then characterized by its
mean function $m(\vec{x})=\mathbb{E}[Z_{\vec{x}}]$, $\vec{x} \in D$, and covariance function 
$k(\vec{x},\vec{y})=\operatorname{Cov}(Z_{\vec{x}}, Z_{\vec{y}})$, 
$\vec{x}, \vec{y} \in D$. 
It is well-known that admissible covariance functions coincide with symmetric positive definite (s.p.d.) kernels on $D\times D$ \cite{ber:tho04}. 

A \emph{multivariate} GRF taking values in $\R^p$ is a collection of $\R^p$-valued random vectors $Z=(Z_{\vec{x}})_{\vec{x}\in D}$ such that $Z_{\vec{x}^{(i)}}^{(j)}$, $1 \leq i \leq n$, $1 \leq j \leq p$, are jointly $np$-variate normally distributed for any $\vec{x}^{(1)}, \dots, \vec{x}^{(n)} \in D$. The distribution of $Z$ is characterized by its $\mathbb{R}^p$-valued mean function and a matrix-valued covariance function $(k_{ij})_{i,j\in \{1, \dots, p\}}$. 
 
In both real- and vector-valued cases (assuming additional technical conditions where necessary)
$k$ governs a number of \textit{pathwise} properties ranging from square-integrability to continuity, differentiability and more; see e.g.\ Section~1.4 of \cite{Adler.Taylor2007} or Chapter~5 of~\cite{scheuerer09} for details.
As we will see in Section~\ref{sec:4}, $k$ actually also governs the FANOVA decomposition of GRF paths $\omega \in \Omega \longrightarrow Z_{\bullet}(\omega) \in \R^{D}$. Before establishing this result, let us first introduce a functional ANOVA decomposition for kernels. 

\section{KANOVA: A kernel ANOVA decomposition}\label{sec:3}

Essentially we apply the $2d$-dimensional version of the decomposition introduced in Section~\ref{sec:2} to $\nu \otimes \nu$-square integrable kernels $k$ (s.p.d. or not). 
From a formal point of view it is more elegant and leads to more efficient notation if we work with the tensor products $T_{\bsu} \otimes T_{\bsv} \colon \mathcal{F} \otimes \mathcal{F} \longrightarrow \mathcal{F} \otimes \mathcal{F}$.
It is well known that $L^2(\nu \otimes \nu)$ and $\mathcal{F} \otimes \mathcal{F}$ are isometrically isomorphic (see \cite{Kree74} for details on tensor products of Hilbert spaces), and  we silently identify them here for simplicity. Then $T_{\bsu} \otimes T_{\bsv} = T^{(1)}_{\bsu} T^{(2)}_{\bsv} = T^{(2)}_{\bsv} T^{(1)}_{\bsu}$, where $T^{(1)}_{\bsu}$, $T^{(2)}_{\bsv} \colon L^2(\nu \otimes \nu) \longrightarrow L^2(\nu \otimes \nu)$ are given by $(T^{(1)}_{\bsu} k)(\vec{x},\vec{y}) = (T_{\bsu}(k(\bullet,\vec{y}))(\vec{x})$ and $(T^{(2)}_{\bsv} k)(\vec{x},\vec{y}) = (T_{\bsv}(k(\vec{x},\bullet))(\vec{y})$. 

\begin{theorem}
\label{kernelANOVA}
Let $k$ be $\nu \otimes \nu$-square integrable. 
\begin{description}
\item[a)] There exist $k_{\bsu, \bsv} \in \mathrm{L}^2(\nu\otimes \nu)$ 
depending solely on $(\vec{x}_{\bsu}, \vec{y}_{\bsv})$ such that 
$k$ can be decomposed in a unique way as $k=\sum_{\bsu, \bsv \subseteq I}k_{\bsu, \bsv}$ 
under the conditions   
\begin{equation}
\forall \bsu, \bsv \subseteq I \ \:
\forall i \in \bsu \ \:
\forall j \in \bsv \ 
\int k_{\bsu, \bsv} \; \nu_{i}(\mathrm{d}x_{i})
= \mathbf{0}
\text{ and }
\int k_{\bsu, \bsv} \; \nu_{j}(\mathrm{d}y_{j})
= \mathbf{0}.
\end{equation}
We have
\begin{equation}
\label{eqkuv}
k_{\bsu, \bsv}(\vec{x}, \vec{y})=\sum_{\bsu'\subseteq \bsu}
\sum_{\bsv'\subseteq \bsv} (-1)^{\vert \bsu\vert+\vert \bsv\vert-\vert\bsu'\vert-\vert \bsv'\vert}\int k(\vec{x},\vec{y}) \; \nu_{-\bsu'}(\mathrm{d}\vec{x}_{-\bsu'})
\; \nu_{-\bsv'}(\mathrm{d}\vec{y}_{-\bsv'}).
\end{equation}
Moreover, $k_{\bsu, \bsv}$ may be written concisely as
$k_{\bsu, \bsv}= [T_{\bsu} \otimes T_{\bsv}] k$.

\item[b)] Suppose that $D$ is compact and $k$ is a continuous s.p.d. kernel. 
Then, for any $(\alpha_{\bsu})_{\bsu \subseteq I} \in \R^{2^{d}}$, the following function is also a s.p.d. kernel:
\begin{equation}
(\vec{x}, \vec{y})\in D\times D \longrightarrow 
\sum_{\bsu \subseteq I}
\sum_{\bsv \subseteq I}
\alpha_{\bsu} \alpha_{\bsv}
k_{\bsu, \bsv}(\vec{x}, \vec{y})
\in \R.
\end{equation}
\end{description}
\end{theorem}

\begin{proof}
The proofs are in the appendix (Section~\ref{ProofkernelANOVA}) to facilitate the reading. 
\end{proof}

\begin{example}[The Brownian kernel]
\label{kernelANOVA_MB}
Consider the covariance kernel $k(x,y)=\min(x,y)$ of the Brownian motion on $D=[0,1]$, and suppose that $\nu$ is the Lebesgue measure. The $k_{\bsu, \bsv}$'s  can then easily be obtained by direct calculation:
$k_{\emptyset, \emptyset}=\frac{1}{3}$,
$k_{\emptyset, \{1\}}(y)=y-\frac{y^{2}}{2}-\frac{1}{3}$, 
$k_{\{1\}, \emptyset}(x)=x-\frac{x^2}{2}-\frac{1}{3}$, 
and 
$k_{\{1\}, \{1\}}(x,y)=\min(x,y)-x+\frac{x^2}{2}-
y+\frac{y^2}{2}+\frac{1}{3}$. 
\end{example}

\begin{example}
\label{ex:ANOVA_TensorProductKernels}
Consider the very common class of tensor product kernels:
$ k(\vec{x},\vec{y}) = \prod_{i = 1}^d  k_i(x_i, y_i) $
where the $k_i$'s are 1-dimensional symmetric kernels.
It turns out that Equation~\eqref{eqkuv} boils down to a sum depending on 1- and 2-dimensional integrals, since
\begin{align}
\label{KANOVATP}
& \int k(\vec{x},\vec{y}) d\nu_{-\bsu}(\vec{x}_{-\bsu})d\nu_{-\bsv}(\vec{y}_{-\bsv}) = \nonumber \\
& \prod_{i \in \bsu \cap \bsv} k_i(x_i, y_i) \:\cdot
\prod_{i \in \bsu \setminus \bsv} \int k_i(x_i, \cdot) d\nu_i \:\cdot
\prod_{i \in \bsv \setminus \bsu} \int k_i(\cdot, y_i) d\nu_i \:\cdot
 \prod_{i \notin \bsu \cup \bsv} \iint k_i d(\nu_i \otimes \nu_i).
\end{align}

\noindent By symmetry of $k$, Equation~\eqref{KANOVATP} solely depends on the integrals $\iint k_i d(\nu_i \otimes \nu_i)$ and integral functions $t \mapsto \int  k_i(\cdot, t) d\nu_i$, $i=1, \dots, d$.
We refer to Section~\ref{sec:9} for explicit calculations using typical $k_i$'s. 
A particularly convenient case is considered next.
\end{example}

\begin{corollary} 
\label{centredANOVA}
Let $k_{i}^{(0)} \colon D_{i} \times D_{i} \longrightarrow \R$ ($1\leq i \leq d$) be argumentwise centred, i.e. such that 
$\int k_{i}^{(0)}(\cdot,t)\mathrm{d}\nu_{i}=\int k_{i}^{(0)}(s,  
\cdot)\mathrm{d}\nu_{i}=0$ for all $i \in I$ and 
$s,t \in D_{i}$, and consider $k(\vec{x},\vec{y})=\prod_{i=1}^{d}(1+k_{i}^{(0)}(x_{i},y_{i}))$. Then the KANOVA decomposition of $k$ consists of the terms 
$[T_{\bsu}\otimes T_{\bsu}]k(\vec{x},\vec{y})
=\prod_{i\in \bsu}k_{i}^{(0)}(x_{i},y_{i})$ and
$[T_{\bsu}\otimes T_{\bsv}]k = \mathbf{0}$ if $\bsu\neq \bsv$.
\end{corollary} 

\begin{remark}
\label{ANOVAkernels}
By taking $k(\vec{x},\vec{y})=\prod_{i=1}^{d}(1+k_{i}^{(0)}(x_{i},y_{i}))$, where $k_i^{0}$ are s.p.d., we recover the so-called ANOVA kernels \cite{Wahba1990, Vapnik1998, jmva_nico}. Corollary~\ref{centredANOVA} guarantees for argumentwise centred $k_{i}^{(0)}$ that the associated $k$ has a simple KANOVA decomposition, with analytically tractable $k_{\bsu, \bsu}$ terms and vanishing $k_{\bsu, \bsv}$ terms (for $\bsu \neq \bsv$). 
\end{remark}

\section{FANOVA decomposition of Gaussian random field paths}\label{sec:4}

Let $Z= (Z_{\vec{x}})_{\vec{x} \in D}$ be a centred GRF with covariance function $k$. To simplify the arguments we make an assumption (for the rest of the article) that is often satisfied in practice: \emph{let $D_i$ be compact subsets of $\mathbb{R}$ and assume that $Z$ has continuous sample paths.} 
The latter can be guaranteed by a weak condition on the covariance kernel; see~\cite{Adler.Taylor2007}, Theorem~1.4.1.
For $r \in \mathbb{N} \setminus \{0\}$ write $C_b(D,\mathbb{R}^r)$ for the space of (bounded) continuous functions $D \to \mathbb{R}^r$ equipped with the supremum norm, and set in particular $C_b(D) = C_b(D,\mathbb{R})$.
We reinterpret $T_{\bsu}$ as maps $C_b(D) \to C_b(D)$, which are still bounded linear operators.
\begin{theorem}
\label{anovaproj_gauss}
The $2^d$-dimensional vector-valued random field
$(Z_{\vec{x}}^{(\bsu)}, \bsu  \subseteq I)_{\vec{x}\in D}$ is
Gaussian, centred, and has continuous sample paths again. Its
matrix-valued covariance function is given by
  \begin{equation}
  \label{covanovaproj}
    \cov(Z^{(\bsu)}_{\vec{x}}, Z^{(\bsv)}_{\vec{y}}) = [T_{\bsu} \otimes T_{\bsv}] k \, (\vec{x}, \vec{y}).
  \end{equation}
\end{theorem}

\begin{example}
Continuing from Example~\ref{kernelANOVA_MB},
let $B=(B_{x})_{x \in [0,1]}$ be Brownian motion on $D=[0,1]$, which is a centred Gaussian random field with continuous paths. Theorem~\ref{anovaproj_gauss} yields that $(T_{\emptyset} B, T_{\{1\}} B) = (\int_{0}^{1} B_{u} \mathrm{d} u,\, B_{x}-\int_{0}^{1} B_{u} \mathrm{d} u)_{x \in D}$ is a bivariate random field on $D$, where $T_{\emptyset} B$ is a $\mathcal{N}(0,1/3)$-distributed random variable, while $(T_{\{1\}} B_{x})$ is a centred Gaussian process with covariance kernel $k_{\{1\}, \{1\}}(x,y)=\min(x,y)-x+\frac{x^2}{2}-y+\frac{y^2}{2}+\frac{1}{3}$. The cross-covariance function of the components is given by $\operatorname{Cov}(T_{\emptyset} B, T_{\{1\}} B_{x})=x-\frac{x^{2}}{2}-\frac{1}{3}$. 
\end{example}

\begin{remark}
\label{KL}
Under our conditions on $Z$ and using the notation from the proof of Theorem~\ref{kernelANOVA}, we have a Karhunen--Lo\`eve expansion
$Z_{\vec{x}}=\sum_{i=1}^{\infty} \sqrt{\lambda_{i}} \varepsilon_{i} \phi_{i}(\vec{x})$, where $\varepsilon = (\varepsilon_i)_{i \in \mathbb{N} \setminus \{0\}}$ is a standard Gaussian white noise sequence and the series converges uniformly (i.e.\ in $C_b(D)$) with probability $1$ (and in $L^2(\Omega)$); for $d=1$ see \cite{Kuelbs1971, Adler.Taylor2007}. 
Thus by the continuity of $T_{\bsu}$, we can expand the projected random field as
\begin{equation}
Z_{\vec{x}}^{(\bsu)}=T_{\bsu}\left(\sum_{i=1}^{\infty} \sqrt{\lambda_{i}} \varepsilon_{i} \phi_{i}(\vec{x})\right)
=\sum_{i=1}^{\infty} \sqrt{\lambda_{i}} \varepsilon_{i} 
T_{\bsu}\left(\phi_{i}\right) (\vec{x}),
\end{equation}
where the series converges uniformly in $\vec{x}$ with probability 1 (and in $L^2(\Omega)$). This is the basis for an alternative proof of Theorem~\ref{anovaproj_gauss}. We can also verify Equation~\eqref{covanovaproj} under these conditions. Using the left/right-continuity of $\operatorname{cov}$ in $L^2(\Omega)$, we obtain indeed 
$\operatorname{cov}\bigl( Z_{\vec{x}}^{(\bsu)}, Z_{\vec{y}}^{(\bsv)}  \bigr) = \sum_{i=1}^{\infty} \lambda_{i} \: T_{\bsu}(\phi_{i})(\vec{x}) \: T_{\bsv}(\phi_{i})(\vec{y})
= k_{\bsu, \bsv} (\vec{x}, \vec{y})$. 
\end{remark}

\begin{corollary}
\label{sparsity}
(a) \ For any $\bsu \subseteq I$ the following statements are equivalent:
\begin{enumerate}
\item[(i)] $T_{\bsu}(k(\bullet,\vec{y})) =\mathbf{0}$ for every $\vec{y} \in D$
\item[(ii)] $[T_{\bsu} \otimes T_{\bsu}]k=\mathbf{0}$ 
\item[(iii)] $[T_{\bsu} \otimes T_{\bsu}]k(\vec{x},\vec{x})=0$ for every $\vec{x} \in D$
\item[(iv)] $\mathbb{P}( Z^{(\bsu)}=\mathbf{0} ) =1$
\end{enumerate}
(b) \ For any $\bsu,\bsv \subseteq I$ with $\bsu \neq \bsv$ the following statements are equivalent:
\begin{enumerate}
\item[(i)] $[T_{\bsu} \otimes T_{\bsv}]k=\mathbf{0}$ 
\item[(ii)] $Z^{(\bsu)}$ and $Z^{(\bsv)}$ are two independent GRFs
\end{enumerate}
\end{corollary}

\begin{remark}
A consequence of Corollary~\ref{sparsity} is that choosing a kernel without $\bsu$ component in GRF-based GSA will lead to a posterior distribution without $\bsu$ component whatever the assimilated data, i.e.\ $\PP(Z^{(\bsu)}=\mathbf{0} \, \vert \, Z_{\vec{x}_1},\ldots,Z_{\vec{x}_n}) = 1$ (a.s.). Indeed, an a.s. constant random element remains a.s. constant under any conditioning.

However, the analogous result does not hold for cross-covariances between $Z^{(\bsu)}$ and $Z^{(\bsv)}$ for $\bsu \neq \bsv$. Let us take for instance $D=[0,1]$, $\nu$ arbitrary, and $Z_{t}=U+Y_{t}$, where $U\sim\mathcal{N}(0,\sigma^2)$ ($\sigma>0$) and $(Y_{t})$ is a centred Gaussian process with argumentwise centred covariance kernel $k^{(0)}$.
Assuming that $U$ and $Y$ are independent, it is clear that $(T_{\emptyset}Z)_{s}=0$ and $(T_{\{1\}}Z)_{t}=Y_{t}$, so  $\operatorname{cov}((T_{\emptyset}Z)_{s}, (T_{\{1\}}Z)_{t})=0$. 
If in addition $Z$ was observed at a point $r\in D$, Equation~\eqref{covanovaproj} yields
$\operatorname{cov}( 
(T_{\emptyset}Z)_{s}, (T_{\{1\}}Z)_{t} | Z_{r})
=( T_{\emptyset} \otimes T_{\{1\}} )(k(\bullet, \star)-k(\bullet, r)k(r, \star)/k(r,r))(s,t)$,
where $k(s,t) = \sigma^2 + k^{(0)}(s,t)$ is the covariance kernel of $Z$.
By Equation~\eqref{eqkuv} we obtain
$\operatorname{cov}( 
(T_{\emptyset}Z)_{s}, (T_{\{1\}}Z)_{t} | Z_{r})
= -\sigma^2 k^{(0)}(t,r)/(\sigma^2 + k^{(0)}(r,r))$, which in general is nonzero. 
\end{remark}

\begin{remark}
Coming back to the ANOVA kernels discussed in Remark~\ref{ANOVAkernels}, Corollary~\ref{sparsity}(b)
implies that for a centred GRF with continuous sample paths and covariance kernel of the form $k(\vec{x},\vec{y})=\prod_{i=1}^{d}(1+k_{i}^{(0)}(x_{i},y_{i}))$, where $k_i^{(0)}$ is argumentwise centred, the FANOVA effects $Z^{(\bsu)}$, $\bsu \subseteq I$, are actually independent.
\end{remark}

To close this section, let us finally touch upon the distribution of Sobol' indices of GRF sample paths, relying on Theorem~\ref{anovaproj_gauss} and Remark~\ref{KL}. 

\begin{corollary}
\label{sobol}
For $\bsu \subseteq I$, $\bsu \neq \emptyset$, we can represent the Sobol' indices of $Z$ as
\begin{equation*}
  S_{\bsu}(Z)= 
\frac{Q_{\bsu}(\varepsilon, \varepsilon)}
{\sum_{\bsv \neq \emptyset} 
Q_{\bsv}(\varepsilon, \varepsilon)},
\end{equation*}
where the $Q_{\bsu}$'s are quadratic forms in a standard Gaussian white noise sequence. In the notation of Remark~\ref{KL},
$
  Q_{\bsu}(\varepsilon, \varepsilon) = \sum_{i=1}^{\infty} \sum_{j=1}^{\infty} \hspace*{-2pt}\sqrt{\lambda_{i}\lambda_{j}} \langle T_{\bsu}\phi_{i}, T_{\bsu}\phi_{j} \rangle \varepsilon_{i} \varepsilon_{j},
$
where the convergence is uniform with probability 1.
\end{corollary}

\begin{remark}
Consider the GRF $Z'=Z-T_{\emptyset}Z$ with Karhunen--Lo{\`e}ve expansion $Z'_{\vec{x}}=\sum_{i=1}^{\infty} \sqrt{\lambda'_{i}} \phi'_{i}(\vec{x}) \varepsilon_i$.
From Equation~\eqref{eq:gensobol} and (the proof of) Corollary~\ref{sobol} we can see that $S_u(Z) = S_u(Z') = \sum_{i,j=1}^{\infty} g'_{ij} \varepsilon_i \varepsilon_j \big/ \sum_{i=1}^{\infty} \lambda'_i \varepsilon_i^2$, where $g'_{ij} = \sqrt{\lambda'_{i}\lambda'_{j}} \langle T_{\bsu}\phi'_{i}, T_{\bsu}\phi'_{j} \rangle$. Truncating both series above at $K \in \mathbb{N}$, applying the theorem in Section~2 of \cite{sawa1978} and then Lebesgue's theorem for $K \to \infty$, we obtain
\begin{equation*}
\begin{split}
 \EE S_{\bsu}(Z) &= \sum_{i=1}^{\infty} g'_{ii} \int_0^{\infty} \Bigl( (1+2 \lambda'_i t)^{3/2} \prod_{l \neq i} (1+2\lambda'_l t)^{1/2} \Bigr)^{-1} \; dt,\\
 \EE S_{\bsu}(Z)^2 &= \sum_{i=1}^{\infty} \sum_{j=1}^{\infty} (g'_{ii} g'_{jj} + 2 g'_{ij}{}^{2}) \int_0^{\infty} t \, \Bigl( (1+2 \lambda'_i t)^{3/2} \prod_{l \not\in \{i,j\}} (1+2\lambda'_l t)^{1/2} \Bigr)^{-1} \; dt.
\end{split}
\end{equation*}
\end{remark}

\section{Making new kernels from old with KANOVA}
\label{sec:5}

While kernel methods and Gaussian process modelling have proven efficient in a number of classification and prediction problems, finding a suitable kernel for a given application is often judged  difficult.  It should simultaneously express the desired features of the problem at hand while respecting positive definiteness, a mathematical constraint that is not straightforward to check in practice. In typical implementations of kernel methods, a few classes of standard stationary kernels are available for which positive definiteness was established analytically based on the Bochner theorem. On the other hand, some operations on kernels are known to preserve positive-definiteness, which enables enriching the available dictionary of kernels notably by multiplication by a positive constant, convex combinations, products and convolutions of kernels, or deformations of the input space. 
The section \textit{Making new kernels from old} of \cite{ras:wil06} (Section 4.2.4) covers a number of such operations. 
We now consider some new ways of creating admissible kernels in the context of the KANOVA decomposition of Section~\ref{sec:3}.
Let us first consider as before some square-integrable symmetric positive definite kernel $k_{\text{old}}$ and take $\bsu \subseteq I$. 

One straightforward approach to create a kernel whose associated Gaussian random field has paths in $\mathcal{F}_{\bsu}$ is then to plainly take the ``simple'' projected kernel
\begin{equation}
\label{projectedk1}
k_{\text{new}} = \pi_{\bsu} k_{\text{old}} \text{ with } 
\pi_{\bsu}=T_{\bsu} \otimes T_{\bsu}.
\end{equation} 
From Theorem~\ref{kernelANOVA}(b) it is clear that such kernels are s.p.d.; however, they will generally not be strictly positive definite. 

Going one step further, one obtains a richer class of $2^{2^{d}}$ positive definite kernels by considering  parts of $\mathcal{P}(I)$,  
and designing kernels accordingly. Taking $U \subset \mathcal{P}(I)$, we obtain a further class of projected kernels as follows:
\begin{equation}
\label{projectedk2}
k_{\text{new}} = \pi_{U} k_{\text{old}} \text{ with } 
\pi_{U} = T_{U} \otimes T_{U} = 
\sum_{\bsu \in U} \sum_{\bsv \in U} 
T_{\bsu} \otimes T_{\bsv}, \text{ where } 
T_{U}= \sum_{\bsu \in U} T_{\bsu}. 
\end{equation} 
The resulting kernel is again s.p.d., which follows from Theorem~\ref{kernelANOVA}(b) by choosing $\alpha_{\bsu} = 1$ if $\bsu \in U$ and  $\alpha_{\bsu}=0$ otherwise. Such a kernel contains not only the covariances induced by the effects associated with the different subsets of $U$, but also cross-covariances between these effects. Finally, another relevant class of positive definite projected kernels can be designed by taking 
\begin{equation}
\label{projectedk2}
k_{\text{new}} = \pi_{U}^{\star} k_{\text{old}} \text{ with } 
\pi_{U}^{\star} 
= \sum_{\bsu \in U}
T_{\bsu} \otimes T_{\bsu}.
\end{equation} 
This kernel corresponds to the one of a sum of independent random fields with same individual distributions as the $Z^{(\bsu)}$ $(\bsu \in U)$.
In addition, projectors of the form $\pi_{U_{1}}$,$\pi^{\star}_{U_{2}}$ ($U_{1}, U_{2} \subset \mathcal{P}(I)$) can be combined (e.g.\ by sums or convex combinations) in order to generate a large class of s.p.d. kernels, as illustrated here and in Section~\ref{sec:6}.

\begin{figure}
\begin{center}
\includegraphics[width=0.9\textwidth]{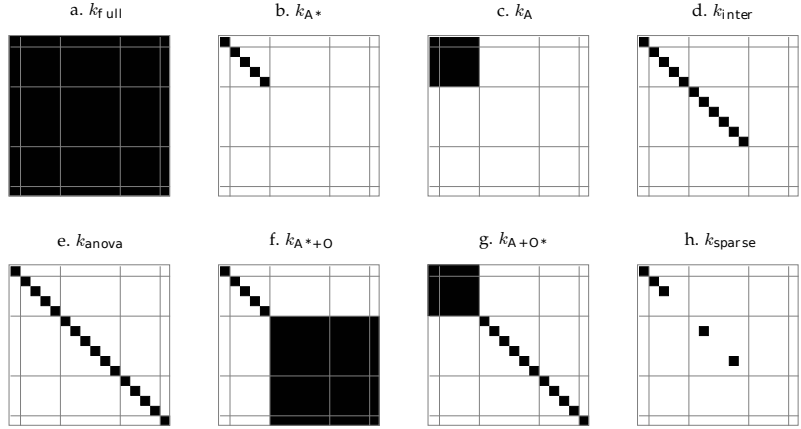} 
\caption{Schematic representations of a reference kernel $k_{\text{full}}$ and various projections or sums of projections. 
The expressions of these kernels are detailed in Section~\ref{sec:6}.}
\label{sparsekernels}
\end{center}
\end{figure}

\begin{example}
\label{orthoadd}
Let us consider $A=\{ \emptyset, \{1\}, \{2\}, \dots, \{d\} \}$ and $O$, the complement of $A$ in $\mathcal{P}(I)$. While $A$ corresponds to the constant and main effects forming the additive component in the FANOVA decomposition, $O$ corresponds to all higher-order terms, referred to as \textit{ortho-additive} component in \cite{Lenz2013}. Taking $\pi_{A}k=(T_{A}\otimes T_{A})k$ amounts to extracting the additive component of $k$ with cross-covariances between the various main effects (including the constant); see Figure~\ref{sparsekernels}(c). On the other hand, $\pi^{\star}_{A}k=\sum_{\bsu \in A}\pi_{\bsu}k$ retains these main effects without their possible cross-covariances; see Figure~\ref{sparsekernels}(b). In the next theorem (proven in \cite{Lenz2013}), analytical formulae are given for $\pi_{A}k$ and related terms for the class of tensor product kernels. 
\end{example}

\begin{theorem}
\label{orthoaddkernel}
Let $D_{i}=[a_{i}, b_{i}]$ ($a_{i}<b_{i}$) and $k=\prod_{i=1}^{d} k_{i}$, where the $k_{i}$ are s.p.d. kernels on $D_{i}$ such that 
$k_{i}(x_{i}, y_{i}) >0$ for all $x_{i}, y_{i} \in D_{i}$.
Then, the additive and ortho-additive components of $k$ with their cross-covariances are given by
\end{theorem}
\begin{equation*}
\begin{aligned}
	(\pi_{A} k) (\vec{x}, \vec{y}) = & \; \frac{a(\vec{x}) a(\vec{y})}{\mathcal{E}} + \mathcal{E} \cdot \sum_{i=1}^d \Bigg( \frac{k_i(x_i, y_i)}{\mathcal{E}_i} - \frac{E_i(x_i) E_i(y_i)}{\mathcal{E}_i^2} \Bigg)\\
(T_{O}\otimes T_{A} k) (\vec{x}, \vec{y}) =  & \; 
(T_{A}\otimes T_{O} k) (\vec{y}, \vec{x}) =
 E(\boldsymbol{x}) \cdot \bigg( 1 - d + \sum_{j=1}^d \frac{k_j(x_j, y_j)}{E_j(x_j)} \bigg) 
-(\pi_{A} k) (\vec{x}, \vec{y})
\\	
	(\pi_{O} k) (\vec{x}, \vec{y}) = & \; k(\vec{x}, \vec{y}) 
	 - (T_{A}\otimes T_{O} k) (\vec{x}, \vec{y}) -(T_{O}\otimes T_{A} k) (\vec{x}, \vec{y})
-(\pi_{A} k) (\vec{x}, \vec{y})
\end{aligned}
\end{equation*}
where $E_i(x_i) = \int_{a_i}^{b_i} { k_i(x_i, y_i) \; dy_i}$, $E(\vec{x}) = \prod_{i = 1}^d E_i(x_i)$, $\mathcal{E}_i = \int_{a_i}^{b_i} { E_i(x_i) \nu_{i}(\mathrm{d}x_i)}$, $\mathcal{E} = \prod_{i = 1}^d \mathcal{E}_i$, and 
$a(\vec{x}) = \mathcal{E} \left( 1 - d + \sum_{i=1}^d \frac{E_i(x_i)}{\mathcal{E}_i} \right)$.

\section{Numerical experiments}\label{sec:6}

\noindent
We now consider a 30 dimensional numerical experiment where we compare the prediction abilities of various sparse kernels obtained from the KANOVA decomposition of a squared-exponential kernel
\begin{equation}
  k_{\text{full}}(\vec{x},\vec{y}) = \exp(-||\vec{x}-\vec{y}||^2), \qquad \vec{x},\vec{y} \in (0,1)^{30}.
\end{equation}

As detailed in the previous sections, $k_{\text{full}}$ can be expanded as a sum of $4^{30}$ terms, and sparsified versions of $k_{\text{full}}$ can be obtained by projections such as in Example~\ref{orthoadd}. We will focus hereafter on seven sub-kernels (all summations are over $\bsu,\ \bsv \subseteq I$):

\begin{equation}
\begin{array}{l l l}
  k_{\text{anova}} = \sum \pi_{\bsu} k & \qquad \qquad &
  k_{A^\star} = \sum_{|\bsu| \leq 1} \pi_{\bsu} k \\
  k_{A} = \sum_{|\bsu| \leq 1} \sum_{ |\bsv| \leq 1} (T_{\bsu}\otimes T_{\bsv}) k & &
  k_{A^\star+O}  = \pi_{O} k + k_{A^\star}  \\
  k_{A+O^\star}  = k_{\text{anova}} - k_{A^\star} + k_{A}& &
  k_{\text{inter}}  = \sum_{|\bsu| \leq 2} \pi_{\bsu} k \\
  k_{\text{sparse}}  =  (\pi_{\emptyset} + \pi_{\{1\}} + \pi_{\{2\}} + \pi_{\{2,3\}} + \pi_{\{4,5\}}) k. &  &\\
\end{array}
\end{equation}

A schematic representation of these kernels can be found in Figure~\ref{sparsekernels}. Note that the tensor product structure of $k_{\text{full}}$ allows to use Theorem~\ref{orthoaddkernel} in order to get more tractable expressions for all kernels above. 
Furthermore, the integrals appearing in the $E_{i}$ and $\mathcal{E}_{i}$ terms can be calculated analytically as detailed in Section~\ref{sec:9}. 

We now compare the predictions obtained by GRF modelling with these kernels on a benchmark of test functions given by sample paths from centred GRFs possessing the same set of kernels ($200$ paths per kernel). 
Whenever the kernel used for prediction is not the same as the one used for simulation, a Gaussian observation noise with variance $\tau^2$ is assumed in the models used in prediction, where $\tau^2$ is chosen so as to reflect the part of  variance that cannot be approximated by the model. 

We consider a training set $X_{\text{train}}$ of $500$ points and a test set $X_{test}$ of $200$ points given by a Latin Hypercube design with optimized maximin criterion~\cite{san:wil:not03,DiceDesign}. The accuracy of the fit is measured using the following criterion :
\begin{equation}
  C = 1 - \frac{\sum (y_i - \hat{y}_i)^2}{\sum y_i^2}
\end{equation}
where $y$ is the vector of the test function values at the test points and $\hat{y}$ is the vector of predicted values. This criterion is equal to one when the prediction error is null and it is equal to zero when the model predicts as bad as the null function.

The values of the criterion for all couples of test functions and models are summarized in Table~\ref{tab:numexp}. Let us stress three important points from these results. 

First, this example illustrates that, unless the correlation range is increased, predicting a GRF
based on $500$ points in dimension $30$ is hopeless 
when the covariance structure is full or close to full  (first four rows of Table~\ref{tab:numexp}) no matter what sub-kernel is chosen for prediction. However, for GRFs with sparser covariance, prediction performances are strongly increased (last four rows of Table~\ref{tab:numexp}). 

Second, still focusing on the four last lines of Table~\ref{tab:numexp}, $k_{\text{inter}}$ seems to offer a nice compromise as it works much better than other sub-kernels on $Z_{\text{inter}}$ and achieves very good performances on the sparser GRF sample paths. Besides this, it is not doing much worse than the best sub-kernels on lines $1$ to $4$. 

Third, we observe that neglecting cross-correlations between blocks has very little or no influence on the results, so that the Gaussian kernel appears to have a structure relatively close to a diagonal one. This point remains to be studied analytically. 

\begin{table}
\begin{center}
\begin{tabular}{l|cccccccc}
& $k_{\text{full}}$ & $k_{\text{anova}}$ & $k_{A^\star+O}$ & $k_{A+O^\star}$ & $k_{\text{inter}}$ & $k_{A^\star}$ & $k_{A}$ & $k_{\text{sparse}}$ \\ \hline
$Z_{\text{full}}$ & 0.06 & 0.05 & 0.06 & 0.05 & 0.05 & 0.03 & 0.04 & 0.01 \\
$Z_{\text{anova}}$ & 0.05 & 0.05 & 0.05 & 0.05 & 0.04 & 0.03 & 0.03 & 0.01 \\
$Z_{A^\star+O}$ & 0.05 & 0.04 & 0.05 & 0.04 & 0.04 & 0.03 & 0.03 & 0.01 \\
$Z_{A+O^\star}$ & 0.06 & 0.06 & 0.06 & 0.06 & 0.05 & 0.04 & 0.04 & 0.01 \\
$Z_{\text{inter}}$ & 0.33 & 0.37 & 0.34 & 0.37 & \textcolor{blue}{0.7} & 0.28 & 0.28 & 0.07 \\
$Z_{A^\star}$ & 0.67 & 0.76 & 0.71 & 0.75 & \textcolor{blue}{0.96} & 
1 & 1 & 0.2 \\
$Z_{A}$ & 0.69 & 0.77 & 0.71 & 0.77 & \textcolor{blue}{0.96} & 1 & 1 & 0.18 \\
$Z_{\text{sparse}}$ & 0.75 & 0.83 & 0.8 & 0.78 & \textcolor{blue}{0.95} & 0.9 & 0.9 & 1 \\ \hline
mean & 0.33 & 0.37 & 0.35 & 0.36 & 0.47 & 0.41 & 0.42 & 0.19
\end{tabular} 
\end{center}
\caption{Average value of $C$ over the 200 replications of the experiment. Lines correspond to classes of test functions (GRF models used for simulation) while columns correspond to the kernels used for prediction. The four last lines of the $k_{\text{inter}}$ column are coloured to highlight the superior performances of that kernel when the class of test functions is as sparse or sparser than $Z_{\text{inter}}$.   
}
\label{tab:numexp}
\end{table}

\section{Conclusion and perspectives}\label{sec:7}

We have proposed an ANOVA decomposition of kernels 
(KANOVA), and shown how KANOVA governs the probability distribution of FANOVA effects of Gaussian random field paths. 
%
This has enabled us in turn to establish that ANOVA kernels correspond to centred Gaussian random fields with independent FANOVA effects, to make progress towards the distribution of Sobol' indices of Gaussian random fields, and also to suggest a number of  operations for making new symmetric positive definite kernels from existing ones. 
Particular cases include the derivation of additive and ortho-additive kernels extracted from tensor product kernels, for which a closed form formula was given. 
Besides this, a $30$-dimensional numerical experiment supports the hypothesis that KANOVA may be a useful approach to designing kernels for high-dimensional kriging, as the performances of the interaction kernel suggest. 
Perspectives include analytically calculating the norm of terms appearing in the KANOVA decomposition to better understand the structure of common GRF models. 
From a practical point of view, a next challenge will be to parametrize decomposed kernels adequately so as to recover from data which terms of the FANOVA decomposition are dominating and to automatically design adapted kernels from this.

\begin{acknowledgement}
The authors would like to thank Dario Azzimonti for proofreading. 
\end{acknowledgement}

\appendix
\section{Proofs}\label{sec:8}

\begin{proof}[Theorem~\ref{kernelANOVA}] 
\label{ProofkernelANOVA}
a) \, 
The first part and the concrete solution~\eqref{eqkuv} follow directly from the corresponding statements in Section~\ref{sec:2}.
Having established \eqref{eqkuv}, it is easily seen that $[T_{\bsu}\otimes T_{\bsv}]k =T^{(1)}_{\bsu} T^{(2)}_{\bsv} k$ coincides with $k_{\bsu,\bsv}$.

b) \, Under these conditions Mercer's theorem applies (see \cite{Steinwart.Scovel2012} for an overview and recent extensions). So there exist a non-negative sequence $(\lambda_{i})_{i \in \mathbb{N}\backslash \{0\}}$, 
and continuous representatives $(\phi_{i})_{i \in \mathbb{N}\backslash \{0\}}$ of an orthonormal basis of $\mathrm{L}^2(\nu)$ such that 
$
k(\vec{x}, \vec{y})=\sum_{i=1}^{\infty} 
\lambda_{i} \phi_{i}(\vec{x}) \phi_{i}(\vec{y}) 
$, \hspace*{0.5pt} $\vec{x}, \vec{y} \in D$,
where the convergence is absolute and uniform. 
Noting that $T_{\bsu}, T_{\bsv}$ are also bounded as operators on continuous functions, applying $T^{(1)}_{\bsu} T^{(2)}_{\bsv}$ from above yields that
\begin{equation}
\sum_{\bsu \subseteq I}
\sum_{\bsv \subseteq I}
\alpha_{\bsu} \alpha_{\bsv}
k_{\bsu, \bsv}(\vec{x}, \vec{y})
=
\sum_{i=1}^{\infty} \lambda_{i} 
\psi_{i}(\vec{x}) \psi_{i} (\vec{y}),
\end{equation}
where $\psi_{i}=
\sum_{\bsu \subseteq I} \alpha_{\bsu} (T_{\bsu}\phi_{i})$. Thus the considered function is indeed s.p.d.
\end{proof}

\begin{proof}[Corollary~\ref{centredANOVA}] Expand the product $\prod_{l=1}^{d}(1+k_{l}^{(0)}(x_{l},y_{l}))$ and conclude by uniqueness of the KANOVA decomposition, noting that $\int \prod_{l\in \bsu}k_{l}^{(0)}(x_{l},y_{l})\nu_i(\mathrm{d}x_i) = \int \prod_{l\in \bsu}k_{l}^{(0)}(x_{l},y_{l})\nu_j(\mathrm{d}y_j) = 0$ for any $\bsu \subseteq I$ and any $i,j \in \bsu$.
\end{proof}

\begin{proof}[Theorem~\ref{anovaproj_gauss}]
Sample path continuity 
implies product-mea\-sur\-abi\-l\-i\-ty of $Z$ and $Z^{(\bsu)}$, respectively, as can be shown by an approximation argument; see e.g.\ Prop.~A.D. in~\cite{schuhmacher2005}. Due to Theorem~3
in~\cite{talagrand1987}, the covariance kernel $k$ is continuous, hence 
$\int_{D} \EE |Z_{\vec{x}}| \, \nu_{-\bsu}(\mathrm{d}\vec{x_{-\bsu}}) \leq (\int_{D} k(\vec{x}, \vec{x}) \, \nu_{-\bsu}(\mathrm{d}\vec{x_{-\bsu}}))^{1/2} < \infty$ for any $\bsu \subseteq I$ and by Cauchy--Schwarz $\int_{D} \int_{D} \EE |Z_{\vec{x}}Z_{\vec{y}}| \, \nu_{-\bsu}(\mathrm{d}\vec{x_{-\bsu}}) \nu_{-\bsv}(\mathrm{d}\vec{y_{-\bsv}}) < \infty$ for any $\bsu,\bsv \subseteq I$. Replacing $f$ by $Z$ in Formula~\eqref{eq2}, taking expectations and using Fubini's theorem yields that $Z^{(\bsu)}$ is centred again. Combining~\eqref{eq2}, Fubini's theorem, and~\eqref{eqkuv} yields
\begin{equation}
\begin{split}
\label{covanovaproj_calc}
  \cov&(Z^{(\bsu)}_{\vec{x}}, Z^{(\bsv)}_{\vec{y}})  \\
    &= \sum_{\bsu'\subseteq \bsu} \sum_{\bsv'\subseteq \bsv} (-1)^{\vert \bsu\vert+\vert \bsv\vert-\vert \bsu'\vert-\vert \bsv'\vert} \cov \biggl( \int Z_{\vec{x}} \; \nu_{-\bsu'}(\mathrm{d}\vec{x}_{-\bsu'}),\, \int Z_{\vec{y}} \;
\nu_{-\bsv'}(\mathrm{d}\vec{y}_{-\bsv'}) \biggr) \\
    &= \sum_{\bsu'\subseteq \bsu} \sum_{\bsv'\subseteq \bsv} (-1)^{\vert \bsu\vert+\vert \bsv\vert-\vert \bsu'\vert-\vert \bsv'\vert}\int \cov(Z_{\vec{x}},Z_{\vec{y}}) \; \nu_{-\bsu'}(\mathrm{d}\vec{x}_{-\bsu'})
\; \nu_{-\bsv'}(\mathrm{d}\vec{y}_{-\bsv'}) \\[1mm]
    &= [T_{\bsu} \otimes T_{\bsv}] k \, (\vec{x}, \vec{y}). \\[-10mm]
    &\hspace*{2mm}
\end{split}
\end{equation}
It remains to show the joint Gaussianity of the $Z^{(\bsu)}$.
First note that $C_b(D,\mathbb{R}^r)$ is a separable Banach space for $r \in \N \setminus \{0\}$. 
We may and do interprete $Z$ as a random element of $C_b(D)$, equipped with the $\sigma$-algebra $\mathcal{B}^{D}$ generated by the evaluation maps $[C_b(D) \ni f \mapsto f(\vec{x}) \in \mathbb{R}]$.
By Theorem~2 in \cite{raj:cam1972} the distribution $\PP Z^{-1}$ of $Z$ is a Gaussian measure on $\bigl( C_b(D),\mathcal{B}(C_b(D)) \bigr)$. Since $T_{\bsu}$ is a bounded linear operator $C_b(D) \to C_b(D)$, we obtain immediately that the ``combined operator'' $\mathfrak{T} \colon C_b(D) 
\to C_b(D,\mathbb{R}^{2^d})$, defined by $(\mathfrak{T}(f))(\vec{x}) = (T_{\bsu}f(\vec{x}))_{\bsu \subseteq I}$, is also bounded and linear. Corollary~3.7 of \cite{tar:vak2007} yields that the image measure $(\PP Z^{-1}) \mathfrak{T}^{-1}$ is a Gaussian measure on 
$C_b(D,\mathbb{R}^{2^d})$.
This means that for every bounded linear operator $\Lambda \colon C_b(D,\mathbb{R}^{2^d}) \to \mathbb{R}$ the image measure $((\PP Z^{-1}) \mathfrak{T}^{-1}) \Lambda^{-1}$ is a univariate normal distribution, i.e.\ $\Lambda(\mathfrak{T}Z)$ is a Gaussian random variable.
Thus, for all $n \in \mathbb{N}$, $\vec{x}^{(i)}\in D$ and $a_i^{(\bsu)} \in \mathbb{R}$, where $1 \leq i \leq n$, $u \subseteq I$, we obtain that $\sum_{i=1}^n \sum_{\bsu \subseteq I} a_i^{(\bsu)} (T_{\bsu} Z)_{\vec{x}^{(i)}}$ is Gaussian by the fact that $[C_b(D) \ni f \mapsto f(\vec{x}) \in \mathbb{R}]$ is continuous (and linear) for every $\vec{x} \in D$. We conclude that $\mathfrak{T}Z = (Z_{\vec{x}}^{(\bsu)}, \bsu  \subseteq I)_{\vec{x}\in D}$ is a vector-valued Gaussian random field.
\end{proof}

\begin{proof}[Corollary~\ref{sparsity}]
\textit{(a)} \, If (i) holds, $[T_{\bsu} \otimes T_{\bsu}]k= T_{\bsu}^{(2)} (T_{\bsu}^{(1)}k)=\mathbf{0}$ by $(T_{\bsu}^{(1)}k)(\bullet,\vec{y}) = T_{\bsu}(k(\bullet,\vec{y}))$; thus (ii) holds. (ii) trivially implies (iii). Statement (iii) means that $\var(Z^{(\bsu)}_{\vec{x}}) = 0$, which implies that $Z^{(\bsu)}_{\vec{x}} = 0$ a.s., since $Z^{(\bsu)}$ is centred. (iv) follows by noting that $\PP(Z^{(\bsu)}_{\vec{x}} = 0)=1$ for all $\vec{x} \in D$ implies $\mathbb{P}( Z^{(\bsu)}=\mathbf{0} ) =1$ by the fact that $Z^{(\bsu)}$ has continuous sample paths and is therefore separable. Finally, (iv) implies (i) because $T_{\bsu}(k(\bullet,\vec{y}))=\cov(Z^{(\bsu)}_{\hspace*{1pt}\text{\scalebox{1}{\raisebox{0pt}{$\bullet$}}}},Z_{\vec{y}})=\mathbf{0}$; compare \eqref{covanovaproj_calc} for the first equality. 

\medskip

\textit{(b)} \, For any $m,n \in \mathbb{N}$ and $\vec{x}_1,\ldots,\vec{x}_m,\vec{y}_1,\ldots,\vec{y}_n \in D$ we obtain by Theorem~\ref{anovaproj_gauss} that $Z^{(\bsu)}_{\vec{x}_1}, \ldots, Z^{(\bsu)}_{\vec{x}_m}, Z^{(\bsv)}_{\vec{y}_1}, \ldots, Z^{(\bsv)}_{\vec{y}_n}$ are jointly normally distributed. Statement~(i) is equivalent to saying that $\cov(Z^{(\bsu)}_{\vec{x}},Z^{(\bsv)}_{\vec{y}}) = 0$ for all $\vec{x}, \vec{y} \in D$. Thus $(Z^{(\bsu)}_{\vec{x}_1}, \ldots, Z^{(\bsu)}_{\vec{x}_m})$ and $(Z^{(\bsv)}_{\vec{y}_1}, \ldots, Z^{(\bsv)}_{\vec{y}_n})$ are independent. Since the sets
\begin{equation}
  \{ (f,g) \in \mathbb{R}^D \times \mathbb{R}^D \colon (f(\vec{x}_1),\ldots,f(\vec{x}_m)) \in A, (g(\vec{y}_1),\ldots,g(\vec{y}_n)) \in B \} 
\end{equation}
with $m,n \in \mathbb{N}$, $\vec{x}_1,\ldots,\vec{x}_m,\vec{y}_1,\ldots,\vec{y}_n \in D$, $A \in \mathcal{B}(\mathbb{R}^m)$, $B \in \mathcal{B}(\mathbb{R}^n)$
generate $\mathcal{B}^D \otimes \mathcal{B}^D$ (and the system of such sets is stable under intersections), statement (ii) follows.
The converse direction is straightforward.
\end{proof}

\begin{proof}[Corollary~\ref{sobol}]
By Remark~\ref{KL}, there is a Gaussian white noise sequence $\varepsilon=(\varepsilon_{i})_{i\in \mathbb{N}\backslash \{0\}}$ such that $Z_{\vec{x}}=\sum_{i=1}^{\infty} \sqrt{\lambda_{i}} \varepsilon_{i} \phi_{i}(\vec{x})$ uniformly with probability 1. From $Z^{(\bsu)}_{\vec{x}}=\sum_{i=1}^{\infty} \sqrt{\lambda_{i}} \varepsilon_{i} T_{\bsu}\phi_{i}(\vec{x})$, we obtain $\| Z^{(\bsu)} \|^2=Q_{\bsu}(\varepsilon, \varepsilon)$ with $Q_{\bsu}$ as defined in the statement. A similar calculation for the denominator of $S_{\bsu}(Z)$ leads to $\sum_{\bsv \neq \emptyset} Q_{\bsv}(\varepsilon, \varepsilon)$, which concludes the proof. 
\end{proof}

\section{Additional examples}\label{sec:9}

Here we give useful expressions to compute the KANOVA decomposition of some famous tensor product kernels with respect to the uniform measure on $[0,1]^{d}$. 
For the sake of simplicity we denote the 1-dimensional kernels that they are based on by $k$ (corresponding to the notation $k_i$ in Example~\ref{ex:ANOVA_TensorProductKernels}). The uniform measure on $[0,1]$ is denoted by $\lambda$.

\begin{example}[Exponential kernel]
\label{ex:exp_kernel}
If $k(x,y) = \exp \left( - \frac{ \vert x-y \vert }{\theta} \right)$, then:
\begin{itemize}
\item $\int_0^1 k(., y) d\lambda = \theta \times
\left[ 2 - k(0, y) - k(1,y) \right]$
\item $\iint_{[0,1]^2} k(.,.) d(\lambda \otimes \lambda)  = 
2 \theta (1 - \theta + \theta e^{-1 / \theta} )$\\
\end{itemize}
\end{example}

\begin{example}[Mat\'ern kernel, $\nu=p+\frac{1}{2}$]
Define for $\nu=p+\frac{1}{2}$ ($p \in \mathbb{N}$):
$$ k(x,y) = \frac{p!}{(2p)!} \sum_{i=0}^p \frac{(p+i)!}{i!(p-i)!} 
 \left( \frac{ \vert x - y \vert}{\theta / \sqrt{8\nu}} \right)^{p-i}
 \times \exp \left(  - \frac{\vert x - y \vert}{\theta / \sqrt{2\nu}}   \right).$$
Then, denoting $\zeta_p = \frac{\theta}{\sqrt{2\nu}}$, we have:
$$\int_0^1 k(., y) d\lambda = \zeta_p \frac{p!}{(2p)!} 
\times \left[ 2c_{p,0} -  A_p \left( \frac{y}{\zeta_p} \right) - A_p \left( \frac{1 - y}{\zeta_p} \right) \right],$$
                           
\noindent where
$$ A_p(u) = \left( \sum_{\ell=0}^p c_{p,\ell} u^\ell \right) e^{-u} 
\quad \textrm{with} \quad c_{p,\ell} = \frac{1}{\ell!} \sum_{i=0}^{p-\ell}{\frac{(p+i)!}{i!} 2^{p-i}}.$$

\noindent
This generalizes Example \ref{ex:exp_kernel}, corresponding to $\nu=1/2$.
Also, this result can be written more explicitly for the commonly selected value $\nu=3/2$ ($p=1, \zeta_1=\theta / \sqrt{3}$):
\begin{itemize}
\item $k(x,y) = \left( 1 + \frac{ \vert x-y \vert }{\zeta_1} \right) \exp \left( - \frac{ \vert x-y \vert }{\zeta_1}\right)$
\item 
$\int_0^1 k(., y) d\lambda = \zeta_1 \times
\left[ 4 - A_1 \left( \frac{y}{\zeta_1} \right) - A_1 \left( \frac{1 - y}{\zeta_1} \right) \right] \vspace{0.2cm}$
with $A_1(u) = (2+u)e^{-u}$
\item $\iint_{[0,1]^2} k(.,.) d(\lambda \otimes \lambda)  = 2\zeta_1 \left[ 
2 - 3 \zeta_1 + (1 + 3 \zeta_1 ) e^{ - 1/\zeta_1 }  \right]$\\ 
\end{itemize}
Similarly, for $\nu=5/2$ ($p=2, \zeta_2=\theta / \sqrt{5}$):
\begin{itemize}
\item $k(x,y) = \left( 1 + \frac{ \vert x-y \vert }{\zeta_2} + \frac{1}{3} \frac{  (x-y)^2 }{(\zeta_2)^2} \right) \exp \left( - \frac{ \vert x-y \vert }{\zeta_2} \right)$
\item $\int_0^1 k(., y) d\lambda = \frac{1}{3} \zeta_2 \times
\left[ 16 - A_2 \left( \frac{y}{\zeta_2} \right) - A_2 \left( \frac{1 - y}{\zeta_2} \right) \right] \vspace{0.2cm}$ 
with $A_2(u) = (8 + 5u + u^2) e^{-u}$
\item $\iint_{[0,1]^2} k(.,.) d(\lambda \otimes \lambda) = \frac{1}{3}\zeta_2 (16 - 30 \,\zeta_2)
+ \frac{2}{3} (1 + 7 \,\zeta_2 + 15 \, (\zeta_2)^2 ) e^{ - 1/\zeta_2 } $
\end{itemize}
\end{example}

\begin{example}[Gaussian kernel]
If $k(x,y) = \exp \left( - \frac{1}{2} \frac{(x-y)^2}{\theta^2} \right)$, then
\begin{itemize}
\item $\int_0^1 k(., y) d\lambda = \theta \sqrt{2\pi} \times
  \left[ \Phi \left(\frac{1-y}{\theta} \right) + \Phi \left(\frac{y}{\theta} \right) - 1 \right] $
\item $\iint_{[0,1]^2} k(.,.) d(\lambda \otimes \lambda) =
2 (e^{-1/(2 \theta^2)} -1 ) + \theta \sqrt{2 \pi} \times \left( 2 \Phi \left(\frac{1}{\theta} \right) - 1 \right)$
\end{itemize}
where $\Phi$ denotes the cdf of the standard normal distribution.
\end{example}

\end{document}